\documentclass[leqno]{amsart}
\usepackage[utf8]{inputenc}
\usepackage{amsmath}
\usepackage{amsfonts}
\usepackage{amssymb}
\usepackage{color}
\usepackage{graphicx}
\usepackage{bbm}
\usepackage{dsfont}
\usepackage{comment}
\usepackage[colorlinks,citecolor=blue,urlcolor=blue,linkcolor = blue,hypertexnames=false]{hyperref}

\numberwithin{equation}{section}
\theoremstyle{plain}
\newtheorem{thm}{Theorem}
\newtheorem{proposition}{Proposition}
\newtheorem{corollary}{Corollary}
\newtheorem{lemma}{Lemma}

\theoremstyle{definition}
\newtheorem{definition}{Definition}

\theoremstyle{remark}
\newtheorem{remark}{Remark}

\newcommand{\II}{\mathcal{I}}
\newcommand{\rr}{\mathbf{r}}

\newcommand{\nb}{m}

\newcommand{\HH}{\mathcal{H}}

\renewcommand{\P}{\mathds{P}}

\allowdisplaybreaks[4]

\begin{document}

\title[Perturbation bounds for eigenspaces]{Perturbation bounds for eigenspaces under a relative gap condition}
\author{Moritz Jirak}
\author{Martin Wahl}
\address{Moritz Jirak, Institut f\"{u}r Mathematische Stochastik, Technische Universit\"{a}t Braunschweig, Universit\"{a}tsplatz 2, 38106 Braunschweig, Germany.}
\email{m.jirak@tu-braunschweig.de}
\address{Martin Wahl, Institut f\"{u}r Mathematik, Humboldt-Universit\"{a}t zu Berlin, Unter den Linden 6, 10099 Berlin, Germany.}
\email{martin.wahl@math.hu-berlin.de}



\begin{abstract}
A basic problem in operator theory is to estimate how a small perturbation effects the eigenspaces of a self-adjoint compact operator. In this paper, we prove upper bounds for the subspace distance, taylored for structured random perturbations. As a main example, we consider the empirical covariance operator, and show that a sharp bound can be achieved under a relative gap condition. The proof is based on a novel contraction phenomenon, contrasting previous spectral perturbation approaches.
\end{abstract}

\maketitle

\section{Introduction}
Let $\Sigma$ be a positive self-adjoint compact operator on a separable Hilbert space $\mathcal{H}$. By the spectral theorem, there exists a sequence $\lambda_1\geq \lambda_2\geq\dots> 0$ of positive eigenvalues (which is either finite or converges to zero), together with an orthonormal system of eigenvectors $u_1,u_2,\dots$ such that $\Sigma =\sum_{i\ge 1}\lambda_i u_i\otimes u_i$, where $ u_i\otimes u_i$ denotes the orthogonal projection onto the span of $u_i$.

Let $\hat \Sigma$ be another positive self-adjoint compact operator on $\mathcal{H}$. We consider $\hat\Sigma$ as a perturbed version of $\Sigma$ and write $E=\hat\Sigma-\Sigma$ for the additive perturbation, which will be thought of as small. Again, by the spectral theorem, there exists a sequence $\hat{\lambda}_1\geq \hat{\lambda}_2\ge\dots> 0$ of positive eigenvalues, together with an orthonormal system of eigenvectors $\hat{u}_1,\hat{u}_2,\dots$ such that $\hat{\Sigma}=\sum_{i\ge 1}\hat{\lambda}_i\hat u_i\otimes\hat u_i$, where $ \hat u_i\otimes\hat u_i $ denotes the orthogonal projection onto the span of $\hat u_i$.

Given a finite subset $\mathcal{I}\subseteq \mathbb{N}$, a basic problem is to bound the distance between the eigenspaces $U_{\II}=\operatorname{span}(u_i:i\in\II)$ and $\hat U_{\II}=\operatorname{span}(\hat u_i:i\in\II)$.  Letting $P_{\II}=\sum_{i\in\II}u_i\otimes u_i$ and $\hat P_{\II}=\sum_{i\in\II}\hat u_i\otimes \hat u_i$ be the orthogonal projections onto $U_{\II}$ and $\hat U_{\II}$, respectively, a natural distance is given by the Hilbert-Schmidt distance $\|\hat P_{\II}-P_{\II}\|_2$, which is equal to $\sqrt{2}$ times the Euclidean norm of the sines of the canonical angles between the corresponding subspaces, see e.g. \cite[Chapter VII.1]{MR1477662}).

A first answer to this problem is given by the Davis-Kahan $\sin \Theta$ theorem, a version of which commonly used in probability and statistics reads as
\begin{align}\label{dk}
\|\hat P_{\II}-P_{\II}\|_2\leq 2 \sqrt{2} \|E\|_2/g_\II, \quad \text{with $g_\II=\min_{i\in\II, j\notin \II}|\lambda_i-\lambda_j|$},
\end{align}
see e.g. \cite{MR0264450,MR1477662,wang_samworth_biometrika_2015}, where \eqref{dk} is proven in \cite{wang_samworth_biometrika_2015} for the case that $\II$ is an interval. Quantity $\|E\|_2$ is often replaced with $\sqrt{|\mathcal{I}|}$ times the operator norm $\|E\|_{\infty}$.

More recently, there has been increasing interest in the case where $\hat{\Sigma}$ arises from $\Sigma$ by random perturbations. In this regard, one of the most prominent examples is the empirical covariance operator, a central object in high-dimensional probability due to its importance in statistics and machine learning. The stochastic nature of this problem leaves room for significant improvements of \eqref{dk}, see e.g. \cite{anderson1963,dauxois_1982,gobet2004,HH09,mas_complex_2014,KL16b,OROURKE201826}. Combined with tools from probability theory, a powerful machinery to derive more precise perturbation results is given by the holomorphic functional calculus for linear operators, see e.g. \cite{K95,C83,HE15}. For instance, assuming that
\begin{equation}\label{Eqholfunccon}
\delta_\II=2\|E\|_\infty/g_\II<1,
\end{equation}
we have the linear expansion
\begin{equation}\label{EqRev3}
\hat P_\II-P_\II=\sum_{i\in\II}\sum_{j\notin \II}\frac{1}{\lambda_i-\lambda_j}(P_iEP_j+P_jEP_i)+S_\II(E)
\end{equation}
with remainder term satisfying $\|S_\II(E)\|_\infty\le |\II|\delta_\II^2/(1-\delta_\II)$, cf. \cite{HE15} or \cite{KL16b}. The first term on the right-hand side of \eqref{EqRev3} represents a first-order approximation for $\hat P_\II-P_\II$. While its Hilbert-Schmidt norm is usually of smaller magnitude than the upper bound in \eqref{dk}, the main drawback of this approach is the requirement \eqref{Eqholfunccon}.

Consider e.g. the case where $\Sigma$ and $\hat\Sigma$ are the population and the empirical covariance operator, respectively (cf. Section \ref{SecAppli} below). Then Lemma \ref{prop:fuknagaev:proj} below shows that under mild assumptions, the squared Hilbert-Schmidt norm of the first-order approximation satisfies
\begin{align}\label{exp:LinTerm}
\sum_{i\in\II}\sum_{j\notin \II}\frac{2\|P_iEP_j\|_2^2}{(\lambda_i-\lambda_j)^2} \leq  C\frac{\log (n)}{n}\sum_{i\in\II}\sum_{j\notin \II}\frac{\lambda_i\lambda_j}{(\lambda_i-\lambda_j)^2}
\end{align}
with high probability, where $C>0$ is a constant. A key feature of reproducing kernel Hilbert spaces and functional data approaches in machine learning and statistics are eigenvalues with an exponential or polynomial decay. For instance, for exponentially decaying eigenvalues and the choice $\II=\{1,\dots,k\}$, $k\geq 1$, the right-hand side of \eqref{exp:LinTerm} is of order $\log (n)/n$, while $\delta_\II$ explodes exponentially in $k$, meaning that \eqref{Eqholfunccon} is quickly violated and the above approach breaks down. It is thus natural to ask whether the first-order approximation in \eqref{EqRev3} still gives accurate bounds (with high probability), if \eqref{Eqholfunccon} is no longer satisfied.

The aim of this paper is to provide an affirmative answer to this question, with a view towards empirical covariance operators. Our main finding is that sharp bounds of the type \eqref{exp:LinTerm} can be derived for $\|\hat P_{\II}-P_{\II}\|_2$, replacing \eqref{Eqholfunccon} with a relative gap condition. This is achieved by exploring a novel contraction phenomenon, bypassing arguments based on the holomorphic functional calculus.

The paper is organised as follows. In Section \ref{SecRelBounds} we derive perturbation bounds in the case that certain relative coefficients (resp. sub-matrices) are bounded. These bounds are deduced from a more general statement, given in Section \ref{SecGenBounds}. Section~\ref{SecECO} presents our main applications to the empirical covariance operator. Besides, our approach can deal with a variety of other structured random perturbations. To illustrate this further, we also discuss random perturbations of low rank matrices in Section \ref{SecOperatornormEx}. Finally, the proof of our main result is given in Section \ref{SecProof}.

\subsection{Further notation}\label{SecNotation}
Let $\langle\cdot,\cdot\rangle$ and $\|\cdot\|$ denote the inner product and the norm on $\mathcal{H}$, respectively. Let $p=\dim \HH$ be the dimension of $\HH$. Abusing notation, an index $i\in\mathbb{N}$ or a set $\II \subseteq \mathbb{N}$ of indices is to be understood as a subset of $\{1,\dots,p\}$ if $p$ is finite. The set $\II^c$ denotes the complement of $\II$ (with respect to $\{1,\dots,p\}$ if $p$ is finite). For $i\geq 1$, we write $P_i= u_i\otimes u_i$ and $\hat P_i=\hat u_i\otimes \hat u_i$. Hence for $\II \subseteq \mathbb{N}$, we have $P_{\II}=\sum_{i\in\II}P_i$ and $\hat P_{\II}=\sum_{i\in\II}\hat P_i$. If $p<\infty$, then we extend the sequence of eigenvalues of $\Sigma$ and $\hat\Sigma$ by adding zeros such that the corresponding eigenvectors form an orthonormal basis of $\mathcal{H}$. If $p=\infty$, then we assume (w.l.o.g.) that the eigenvectors $u_1,u_2,\dots$ form an orthonormal basis of $\mathcal{H}$. Thus we always have $\sum_{i\geq 1}P_i=I$. Given a bounded (resp. Hilbert-Schmidt) operator $A$ on $\HH$, we write $\|A\|_\infty$ (resp. $\|A\|_2$) for the operator norm (resp. the  Hilbert-Schmidt norm). Given a trace class operator $A$ on $\HH$, we denote the trace of $A$ by $\operatorname{tr}(A)$.

\section{Main results}\label{SecMainRes}

\subsection{Sharp relative perturbation bounds}\label{SecRelBounds}
We assume throughout Section \ref{SecRelBounds} that the eigenvalues $(\lambda_i)$ are strictly positive and summable,  meaning that $\Sigma$ is a strictly positive, self-adjoint trace class operator. We begin with introducing the crucial relative eigenvalue separation measure.
\begin{definition}\label{DefRelRank} For a subset $\mathcal{I}\subseteq \mathbb{N}$, we define
\begin{equation*}
\rr_{\II}(\Sigma)=\sum_{i\in\II}\frac{\lambda_i}{\min_{j\notin \II}|\lambda_i-\lambda_j|}+\sum_{j\notin\II}\frac{\lambda_j}{\min_{i\in \II}|\lambda_j-\lambda_i|}.
\end{equation*}
\end{definition}
The quantity $\rr_{\II}(\Sigma)$ measures in a weighted way how well the eigenvalues in $(\lambda_i)_{i\in\II}$ are separated from the rest of the spectrum. Let us consider two examples. First, for $k\geq 1$, we have
\begin{equation}\label{EqCCondSE}
\rr_{\{i:\lambda_i=\lambda_k\}}(\Sigma)=\frac{m_k\lambda_k}{g_k}+\sum_{j:\lambda_j\neq \lambda_k}\frac{\lambda_j}{|\lambda_j-\lambda_k|}
\end{equation}
with multiplicity $m_k=|\{i:\lambda_i=\lambda_k\}|$ and gap $g_k=\min_{i:\lambda_i\neq \lambda_k}|\lambda_i-\lambda_k|$. Second, for $k\geq 1$, we have
\begin{equation}\label{EqAssReconstr}
\rr_{\{1,\dots,k\}}(\Sigma)=\sum_{i\leq k}\frac{\lambda_i}{\lambda_i-\lambda_{k+1}}+\sum_{j>k}\frac{\lambda_j}{\lambda_k-\lambda_j}.
\end{equation}
The eigenvalue expressions in \eqref{EqCCondSE} and \eqref{EqAssReconstr} can be easily evaluated if the $\lambda_j$ have exponential or polynomial decay, or more generally, under some convexity conditions (cf. Section \ref{SecAppli}). We now state our first main result.
\begin{thm}\label{ThmSinTheta}
Let $\mathcal{I}\subseteq \mathbb{N}$ be finite. Suppose that there is a real number $x>0$ such that for all $i,j\geq 1$,
\begin{equation}\label{EqCoeffRelBound}
\|P_iEP_j\|_2\leq x\sqrt{\lambda_i\lambda_j}.
\end{equation}
If
\begin{equation}\label{EqCCondSD1}
\rr_{\II}(\Sigma)\leq 1/(8x),
\end{equation}
then we have
\begin{equation}\label{EqLocalSinTheta}
\|\hat P_{\mathcal{I}}- P_{\mathcal{I}}\|_2^2\leq 16x^2\sum_{i\in\mathcal{I}}\sum_{j\notin\mathcal{I}}\frac{\lambda_i\lambda_j}{(\lambda_i-\lambda_j)^2}.
\end{equation}
\end{thm}

\begin{remark}
The numerical constants in \eqref{EqCCondSD1} and \eqref{EqLocalSinTheta} are selected for convenience.
\end{remark}

\begin{remark}
Motivated by the empirical covariance operator, Theorem \ref{ThmSinTheta} considers a perturbation problem where the perturbation $E$ is related to $\Sigma$. There is, however, also a connection to numerical analysis. If $p$ is finite, then the real number $x$ can be chosen as the least upper bound of the absolute values of the $\langle u_i, Eu_j\rangle/\sqrt{\lambda_i\lambda_j}$, $i,j\in\{1,\dots,p\}$. These quantities are the coefficients of the so called relative perturbation $\Sigma^{-1/2}E\Sigma^{-1/2}$ with respect to the eigenvectors of $\Sigma$. The latter matrix plays a prominent role in relative perturbation theory, see e.g. \cite{I98,I00}. The novel ingredient of Theorem \ref{ThmSinTheta} is Condition \eqref{EqCCondSD1}, ensuring that sharp bounds can derived. Indeed, \eqref{EqLocalSinTheta} gives the size of the squared Hilbert-Schmidt norm of the first-order approximation in \eqref{EqRev3}, provided that the bounds in \eqref{EqCoeffRelBound} are sufficiently tight.
\end{remark}

\begin{remark}\label{rem:first:order}
An inspection of the proof shows that the inequality
\begin{equation}\label{EqLocalSinTheta1}
\|\hat P_{\mathcal{I}}- P_{\mathcal{I}}\|_2^2\leq 8 \sum_{i\in\mathcal{I}}\sum_{j\notin\mathcal{I}}\frac{\|P_iEP_j\|_2^2}{(\lambda_i-\lambda_j)^2}+512x^2\rr^2_{\II}(\Sigma)\sum_{i\in\mathcal{I}}\sum_{j\notin\mathcal{I}}\frac{\lambda_i\lambda_j}{(\lambda_i-\lambda_j)^2}
\end{equation}
holds, from which \eqref{EqLocalSinTheta} follows by inserting \eqref{EqCoeffRelBound} and \eqref{EqCCondSD1}.
\end{remark}

Next, we state the following generalization of Theorem \ref{ThmSinTheta}, more suitable
for infinite-dimensional Hilbert spaces:
\begin{thm}\label{ThmSinThetaExt} Let $\II\subseteq \mathbb{N}$ be finite. Write $\II=\dot\cup_{r\leq m}\II_r$ such that for all $r\leq m$ and all $i,j\in\II_r$ we have $\lambda_i= \lambda_j$. Let $\II'\subseteq\mathbb{N}$ be another finite subset such that $|\lambda_i-\lambda_j|\geq \lambda_i/2$ for all $i\in\II$ and all $j\notin\II'$. Write $\II'\setminus \II=\dot\cup_{m<r\le m+n}\II_r$ such that for all $m<r\leq m+n$ and all $i,j\in\II_r$ we have $\lambda_i= \lambda_j$. Let $\II_{m+n+1}=\II'^c$. Suppose that there is a real number $x>0$ such that for all $r,s\leq m+n+1$,
\begin{equation}\label{EqRelCoeffExt}
\|P_{\II_r}EP_{\II_s}\|_2\leq x\sqrt{\sum_{i\in\II_r}\lambda_i}\sqrt{\sum_{j\in \II_s}\lambda_j}.
\end{equation}
If $\rr_{\II}(\Sigma)\leq 1/(8x)$, then \eqref{EqLocalSinTheta} holds with the constant $16$ replaced by $64$.
\end{thm}
Theorem \ref{ThmSinThetaExt} reveals that it actually suffices to have adequate bounds for certain sub-matrices corresponding to the same eigenvalue, and that the far away part represented by $\II'^c$ can be dealt with separately. Note that \eqref{EqRelCoeffExt} follows from \eqref{EqCoeffRelBound}, as can be seen by squaring out the Hilbert-Schmidt norm.

A prototype of a relative perturbation is given through $\hat\Sigma=\Sigma+x (v\otimes v)$ with $x>0$ and $v=\sum_{i\geq 1}\sqrt{\lambda_i}u_i$. In this case, \eqref{EqRelCoeffExt} holds with equality. In contrast, the relative bound \cite[Theorem 3.6]{I00}, for instance, contains the operator norm of the relative perturbation (equals $xp$), and thus becomes useless in higher dimensions.

\subsection{A general perturbation bound}\label{SecGenBounds}
We now present a more general perturbation bound and show how it implies Theorems 2 and 3. In order to deal with the different assumptions on certain coefficients (resp. sub-matrices) of $E$ from the last section, we introduce some flexibility with respect to the structure of the perturbation. In addition, we show that the relative gap conditions in Theorems 2 and 3 can be weakened by also incorporating the operator norm.
\begin{thm}\label{ThmSinThetaAbstr}
Let $\II\subseteq \mathbb{N}$ be a finite subset and let $\{\II_1,\dots,\II_{\nb}\}$ be a partition of $\II$ (meaning that $\II_1,\dots,\II_{\nb}$ are non-empty, disjoint subsets of $\II$ whose union is equal to $\II$). Let $\{\II_{{\nb}+1},\II_{{\nb}+2},\dots\}$ be a (possibly finite) partition of $\II^c$ into intervals.  Let $(a_r)$ and $(b_r)$ be sequences of non-negative real numbers such that for all $r,s\geq 1$,
\begin{equation}\label{EqBoundCoeffBlock}
\|P_{\II_r}EP_{\II_s}\|_\infty\leq \max(\sqrt{a_rb_s},\sqrt{b_ra_s}),\qquad\|P_{\II_r}EP_{\II_s}\|_2\leq \sqrt{b_rb_s}.
\end{equation}
Suppose that
\begin{equation}\label{EqCCondBlock}
\bigg(\sum_{r\geq 1}\frac{a_r}{g_r}\bigg)\bigg(\sum_{r\geq 1}\frac{b_r}{g_r}\bigg)\leq 1/64
\end{equation}
with $g_r=\min_{i\in\II_r,j\notin \II}|\lambda_i-\lambda_j|$ for $r\leq \nb$ and $g_r=\min_{j\in\II_r,i\in\II}|\lambda_i-\lambda_j|$ otherwise.
Then we have
\begin{align}
\|\hat P_{\II}- P_{\II}\|_2^2&\leq 12\sum_{r> \nb}\sum_{s\leq \nb}\frac{b_rb_s}{g_{r,s}^2}+256\bigg(\sum_{r\geq 1}\frac{b_r}{g_r}\bigg)^2\sum_{r> \nb}\sum_{s\leq \nb}\frac{a_rb_s}{g_{r,s}^2}\label{EqSinThetaGeneral}
\end{align}
with $g_{r,s}=\min_{i\in\II_r,j\in\II_s}(\lambda_{i}- \lambda_j)^2$.

In particular, if \eqref{EqBoundCoeffBlock} holds with $a_r=b_r$ and if $\sum_{r\ge 1}b_r/g_r\le 1/8$, then we have
\begin{align}\label{EqSinThetaGeneralCons}
\|\hat P_{\II}- P_{\II}\|_2^2&\leq 16\sum_{r> \nb}\sum_{s\leq \nb}\frac{b_rb_s}{g_{r,s}^2}.
\end{align}
\end{thm}
\begin{remark}
From \eqref{EqCCondBlock}, it follows that $(b_r)$ is summable. Combining this with Assumption \eqref{EqBoundCoeffBlock}, we see that $E$ has to be Hilbert-Schmidt.
\end{remark}
\begin{remark}
Theorem \ref{ThmSinThetaAbstr} includes a version of the Davis-Kahan $\sin \Theta$ theorem. Indeed, the simple choice $\II_1=\II$, $\II_2=\II^c$ and $a_r=\|E\|_\infty^2/\|E\|_2$, $b_r=\|E\|_2$, $r=1,2$, leads to $\|\hat P_{\II}- P_{\II}\|_2\le 4\|E\|_2/g_\II$, provided that $\|E\|_\infty/g_\II\le 1/16$, with $g_\II=\min_{i\in\II, j\notin \II}|\lambda_i-\lambda_j|$. One advantage of the Davis-Kahan $\sin\Theta$ theorem is that it depends only on a small number of parameters: this version, for instance, shows that the sensitivity of $\hat P_{\II}- P_{\II}$ can be described by the size of the perturbation relative to the gap $g_\II$. Our main objective is to go beyond this simple worst-case scenario using only a single gap. This corresponds to choosing finer partitions. In the extreme case where both partitions consist of singletons, the bound reflects the magnitude of the first-order approximation given in \eqref{EqRev3}, and involves gaps between all relevant eigenvalues.
\end{remark}
\begin{remark}
Assumption \eqref{EqBoundCoeffBlock} is designed to deal with random perturbations. While it might be difficult to check \eqref{EqBoundCoeffBlock} for a given fixed $\Sigma$ and $\hat \Sigma$, we show in Section~\ref{SecAppli} that it holds with high probability for a variety of structured random perturbations. In this respect, note that Assumption \eqref{EqBoundCoeffBlock} allows for some flexibility when bounding $P_{\II_r}EP_{\II_s}$. Observe that since $\|\cdot\|_{\infty} \leq \|\cdot\|_2$, the second condition implies the first if $a_r=b_r$ for all $r\geq 1$. If, however, significantly better bounds for the operator norm are available, see e.g. \cite{V12,MR3185193,vH15}, we may select $a_r \ll b_r$, yielding much weaker conditions in \eqref{EqCCondBlock}.
\end{remark}
\begin{remark}\label{RemGenHM}
If $p=\dim\mathcal{H}<\infty$, then Theorem \ref{ThmSinThetaAbstr} holds for all self-adjoint operators $\Sigma$ and $\hat\Sigma$. Indeed, we can always find a real number $y>0$ such that $\Sigma+yI$ and $\hat{\Sigma}+yI$ are positive, and the claim follows because eigenvectors, gaps, and $E$ are invariants of this transformation. If $p=\infty$, then positiveness is also not necessary, but the statement and its proof are notationally more involved.
\end{remark}

We conclude this section by showing how Theorems \ref{ThmSinTheta} and \ref{ThmSinThetaExt} can be obtained by an application of Theorem \ref{ThmSinThetaAbstr}.
\begin{proof}[Proof of Theorems \ref{ThmSinTheta} and \ref{ThmSinThetaExt}]
In order to obtain Theorem \ref{ThmSinTheta}, take partitions of $\II$ and $\II^c$ consisting  of singletons. Choose $a_j=b_j=x\lambda_j$, with $x$ from \eqref{EqCoeffRelBound}. Then \eqref{EqBoundCoeffBlock} holds because $\|P_{i}EP_{j}\|_\infty=\|P_{i}EP_{j}\|_2$ and \eqref{EqCCondBlock} coincides with \eqref{EqCCondSD1}. Thus \eqref{EqLocalSinTheta} follows from \eqref{EqSinThetaGeneralCons}. Regarding Theorem \ref{ThmSinThetaExt}, the partition is already given. In addition, for $r\leq m+n+1$, set $a_r=b_r=x\sum_{i\in\II_r}\lambda_i$. Then it is easy to see that \eqref{EqBoundCoeffBlock} and \eqref{EqCCondBlock} are implied by \eqref{EqRelCoeffExt} and \eqref{EqCCondSD1}, respectively, and the claim follows from \eqref{EqSinThetaGeneralCons}, using that by construction of $\II'$,
\[
\frac{x^2\lambda_i\sum_{j\in {\II'}^c}\lambda_j}{\min_{j\in {\II'}^c}(\lambda_i-\lambda_j)^2}\leq 4x^2\sum_{j\in {\II'}^c}\frac{\lambda_i\lambda_j}{\lambda_i^2}\leq 4x^2\sum_{j\in {\II'}^c}\frac{\lambda_i\lambda_j}{(\lambda_i-\lambda_j)^2}
\]
for all $i\in\II$.
\end{proof}

\section{Applications}\label{SecAppli}
\subsection{Empirical covariance operators}\label{SecECO}
Let us discuss applications of our main result to the empirical covariance operator.
Let $X$ be a random variable taking values in $\mathcal{H}$. We suppose that $X$ is centered and strongly square-integrable, meaning that $\mathbb{E} X =0$ and $\mathbb{E}\|X\|^2<\infty$. Let $\Sigma =\mathbb{E} X\otimes X$ be the covariance operator of $X$, which is a positive, self-adjoint trace class operator, see e.g. \cite[Theorem 7.2.5]{HE15}. For $j\geq 1$, let $\eta_j=\lambda_j^{-1/2}\langle u_j, X\rangle$ be the $j$-th Karhunen-Lo\`{e}ve coefficient of $X$. Let $X_1,\dots,X_n$ be independent copies of $X$ and let
\begin{equation*}
\hat{\Sigma}=\frac{1}{n}\sum_{l=1}^nX_l\otimes X_l
\end{equation*}
be the empirical covariance operator. Combining Theorem \ref{ThmSinThetaExt} with concentration inequalities, we get:

\begin{thm}\label{ThmCovOp} In the above setting, suppose that
\begin{equation}\label{EqMomentAss}
\sup_{j \geq 1}\mathbb{E}|\eta_j|^q \leq C_{\eta}
\end{equation}
for $q>4$ and a constant $C_{\eta} > 0$. {Then there are constants $c_1,C_1 > 0$ depending only on $C_{\eta}$  and $q$,  such that for all $k, k_0\geq 1$ with $\lambda_{k_0}\leq \lambda_k/2$ and all $t \geq 1$ satisfying}
\begin{equation}\label{EqCCond}
\frac{t}{\sqrt{n}}\bigg(\sum_{i\leq k}\frac{\lambda_i}{\lambda_i-\lambda_{k+1}}+\sum_{j>k}\frac{\lambda_j}{\lambda_k-\lambda_j}\bigg)\leq c_1,
\end{equation}
we have
\begin{align}
& \mathbb{P}\Big(\|\hat P_{\{1,\dots,k\}}-P_{\{1,\dots,k\}}\|_2^2 > \frac{C_1t^2}{n} \sum_{i\leq k}\sum_{j>k}\frac{\lambda_i\lambda_j}{(\lambda_i-\lambda_j)^2}\Big)\leq  k_0^2 \Big(\frac{n^{1-q/4}}{t^{q/2}} +   \exp(-t^2)\Big)\label{EqCorCovOpEq}.
\end{align}
\end{thm}


\begin{remark}
Theorem \ref{ThmCovOp} gives useful bounds for $t \geq \sqrt{\log k_0}$. Corresponding bounds for $t <  \sqrt{\log k_0}$ can be obtained using Remark \ref{rem:first:order}, we omit the details. In \eqref{EqCorCovOpEq}, $k_0$ can be replaced by the number of distinct eigenvalues with indices smaller than or equal to $k_0$.
\end{remark}

\begin{remark}
In the literature it is often assumed that the $\eta_j$ are independent and satisfy some moment growth condition, see e.g. \cite{mas_complex_2014}, or that $X$ is sub-Gaussian or even Gaussian, see e.g. \cite{KL14,KL16b}. In contrast, we only need the existence of a uniform moment bound on the $\eta_j$ of order $q > 4$. In fact, since our bounds are based on the Fuk-Nagaev inequality, we expect our moment assumptions to be minimal. Despite this generality, we obtain sharp results, capable of serving as a new tool in functional PCA or kernel PCA (cf. \cite{HE15,MR2920735,SS2001}).
\end{remark}

For $t=\sqrt{\log n}$, Theorem \ref{ThmCovOp} gives the same high probability bound as in \eqref{exp:LinTerm}. Following \cite{JW18}, we call \eqref{EqCCond} a relative rank condition, in contrast to the effective rank condition introduced in \cite{KL16b}, where the latter is based on \eqref{Eqholfunccon} and the concentration inequality in \cite{KL14,adamczak2015}. The relative rank condition can be easily verified for exponentially or polynomially decaying eigenvalues. For instance, for polynomially decaying eigenvalues, the eigenvalue expressions in \eqref{EqCCond} and \eqref{EqCorCovOpEq} are of order $k\log(k)$ and $k^2\log(k)$, respectively (cf. \cite[Lemma 7.13]{M16}). By a monotonicity argument, the same is true if the decay only holds approximately:

\begin{corollary}\label{CorEDPD}
Grant Assumption \eqref{EqMomentAss}. If for some $\alpha>0$, $\lambda_j\geq j^{-\alpha-1}$ for $j\leq d$ and $\lambda_j\leq j^{-\alpha-1}$ for $j\geq  d$, then there are constants $c_1,C_1 > 0$ depending only on $\alpha$, $C_{\eta}$, and $q$ such that for all $k\geq 2$ and all $t\geq 1$ satisfying $t k\log (k)\leq c_1\sqrt{n}$, we have
\begin{equation*}
 \mathbb{P}\Big(\|\hat P_{\{1,\dots,k\}}-P_{\{1,\dots,k\}}\|_2^2 > \frac{C_1t^2k^2\log (k)}{n}\Big)\leq  k^2\big(n^{1-q/4}t^{-q/2} +   \exp(-t^2)\big).
\end{equation*}
Moreover, if $\lambda_j\geq \exp(-\alpha j)$ for $j\leq k$ and $\lambda_j\leq \exp(-\alpha j)$ for $j\geq k$, then for all $k\geq 1$ and all $t\geq 1$ satisfying $tk\leq c_1\sqrt{n}$, we have
\begin{equation*}
 \mathbb{P}\Big(\|\hat P_{\{1,\dots,k\}}-P_{\{1,\dots,k\}}\|_2^2 > \frac{C_1t^2}{n}\Big)\leq  k^2\big(n^{1-q/4}t^{-q/2} +   \exp(-t^2)\big).
\end{equation*}
\end{corollary}

Relative perturbation bounds for the empirical covariance operator have recently attracted attention in the literature. In \cite{mas_complex_2014,M16,JW18}, using different arguments, the special case $\II = \{i\}$ was treated. The general case is more complicated. For instance, \cite{mas_complex_2014} combines the holomorphic functional calculus with a nice normalization argument to go beyond the standard approach outlined in the introduction. They were, however, not able to obtain the sharp leading term in \eqref{EqCorCovOpEq} by their method of proof, and require much stronger probabilistic conditions. 
For further, related results see also \cite{RW17}, where instead of the subspace distance the excess risk is treated. Several types of relative perturbations have also been investigated in the deterministic case, see e.g. ~\cite{I00}. However, designed for a different purpose, they give significantly inferior results when applied to the empirical covariance operator.

Let us now turn to the proof of Theorem \ref{ThmCovOp}. The following lemma provides the necessary concentration inequality needed to deal with Condition \eqref{EqRelCoeffExt}.

\begin{lemma}\label{prop:fuknagaev:proj}
In the above setting, suppose that \eqref{EqMomentAss} holds. Let $\mathcal{I}, \mathcal{J} \subseteq \mathbb{N}$. Then there is a constant $C_1 > 0$ depending only on $C_\eta$ and $q$, such that for $t \geq 1$,
\begin{align*}
\P\Big(\frac{\|P_{\mathcal{I}} E P_{\mathcal{J}}\|_2}{(\sum_{i \in \mathcal{I}}\sum_{j \in \mathcal{J}}  \lambda_i \lambda_j)^{1/2}} \geq \frac{C_1t}{\sqrt{n}} \Big) \leq \frac{n}{(\sqrt{n}t)^{q/2}} + \exp(-t^2) .
\end{align*}
\end{lemma}

Theorem \ref{ThmCovOp} is now an immediate consequence of Theorem \ref{ThmSinThetaExt} applied with $x=C_1t/\sqrt{n}$, the union bound, and Lemma \ref{prop:fuknagaev:proj}. Lemma \ref{prop:fuknagaev:proj} itself follows from \cite[Theorem 3.1]{MR2434306}, a Fuk-Nagaev type inequality in Banach space. For the sake of completeness, we describe the relevant computations below.

\begin{proof}[Proof of Lemma \ref{prop:fuknagaev:proj}]
Observe that
\begin{align*}
n P_{\mathcal{I}} E P_{\mathcal{J}} &= \sum_{l = 1}^n \sum_{i \in \mathcal{I}} \sum_{j \in \mathcal{J}} (\langle X_l, u_i \rangle \langle X_l, u_j \rangle - \delta_{ij} \sqrt{\lambda_i \lambda_j} ) u_i \otimes u_j=: \sum_{l = 1}^n Z_l.
\end{align*}
By Jensen's inequality and \eqref{EqMomentAss}, we have
\begin{align*}
 \Big(\mathbb{E}\big\|\sum_{l = 1}^n Z_l \big\|_2\Big)^2 \leq C_2 n  \sum_{i \in \mathcal{I}}\sum_{j \in \mathcal{J}} \lambda_i  \lambda_j.
\end{align*}
Similarly, we obtain for $l =1,\dots,n$,
\begin{align*}
\mathbb{E} \|Z_l\|_2^{q/2} \leq C_3 \Big(\sum_{i \in \mathcal{I}}\sum_{j \in \mathcal{J}} \lambda_i  \lambda_j \Big)^{q/4}.
\end{align*}
Finally, for any Hilbert-Schmidt operator $f$ on $\mathcal{H}$ with $\|f\|_2 \leq 1$, we have
\begin{align*}
\mathbb{E} \sum_{l = 1}^n \langle f, Z_l \rangle_2^2\leq n\mathbb{E}\sum_{i \in \mathcal{I}}\sum_{j \in \mathcal{J}} (\langle X, u_i \rangle \langle X, u_j \rangle - \delta_{ij}\sqrt{\lambda_i \lambda_j})^2 \leq C_2 n \sum_{i \in \mathcal{I}}\sum_{j \in \mathcal{J}}  \lambda_i \lambda_j,
\end{align*}
as can be seen by the Cauchy-Schwarz inequality and \eqref{EqMomentAss}. Here $\langle\cdot,\cdot\rangle_2$ denotes the Hilbert-Schmidt scalar product. Lemma \ref{prop:fuknagaev:proj} now follows from \cite[Theorem 3.1]{MR2434306} applied to the Hilbert space of all Hilbert-Schmidt operators on $\mathcal{H}$.
\end{proof}

\subsection{Using the operator norm}\label{SecOperatornormEx}
Let us give two simple examples showing how one can benefit from having two different norms in Condition \eqref{EqBoundCoeffBlock}.

\subsubsection*{Random perturbation of a low-rank matrix}
Let $\HH=\mathbb{R}^p$ and let $\Sigma=\sum_{i\leq k}\lambda_iu_iu_i^T$ be a symmetric matrix with $\lambda_1\geq \dots\geq \lambda_k>0$ and $u_1,\dots,u_k$ orthonormal system in $\mathbb{R}^p$. Moreover, let $\hat\Sigma=\Sigma+\epsilon\xi$, where $\epsilon>0$ and $\xi=(\xi_{ij})_{1\leq i,j\leq p}$ is a GOE matrix, i.e. a symmetric random matrix whose upper triangular entries are independent zero mean Gaussian random variables with $\mathbb{E}\epsilon_{ij}^2=1$ for $1\leq i<j\leq p$ and $\mathbb{E}\epsilon_{ii}^2=2$ for $i=1,\dots,p$. Then Theorem \ref{ThmSinThetaAbstr} yields that for all $t\geq 1$, with probability at least $1-18\exp(-c_1t)$,
\begin{equation}\label{EqGOEPert}
\|\hat P_{1}-P_{1}\|_2^2\leq C_1\bigg(\frac{\epsilon^2tk}{(\lambda_1-\lambda_2)^2}+\frac{\epsilon^2t(p-k)}{\lambda_1^2}+\frac{\epsilon^4t^2(p-k)^2}{\lambda_1^2(\lambda_1-\lambda_2)^2}\bigg),
\end{equation}
where $c_1,C_1>0$ are absolute constants. In comparison, the Davis-Kahan $\sin\Theta$ theorem in \eqref{dk} with $\|E\|_2$ replaced by $\sqrt{|I|}\|E\|_\infty$ yields a bound of order $\epsilon^2p/(\lambda_1-\lambda_2)^2$, which is inferior to \eqref{EqGOEPert} for $k$ smaller than $p$ and $\lambda_1-\lambda_2$ smaller than $\lambda_1$. The bound in \eqref{EqGOEPert} can be compared to \cite[Theorem 8]{VVU2011} and \cite[Remark 15]{OROURKE201826}, where a structurally similar third term appears. Note that this term can be avoided under an additional relative gap condition, by applying \eqref{EqSinThetaGeneralCons} instead of \eqref{EqSinThetaGeneral}.

Let us deduce \eqref{EqGOEPert} from \eqref{EqSinThetaGeneralCons}, using also Remark \ref{RemGenHM}. Since $\xi$ is invariant under orthogonal transformations, we may assume that $u_i$ is the $i$-th standard basis vector in $\mathbb{R}^{p}$. Hence, Condition \eqref{EqBoundCoeffBlock} is dealing with sub-blocks of $E=\epsilon\xi$. We choose $\II_1=\{1\}$, $\II_2=\{2,\dots,k\}$, and $\II_3=\{k+1,\dots,p\}$.
Applying concentration results for the operator norm of random matrices (e.g. \cite[Theorem 5.6]{MR3185193} and \cite[Theorem 1]{latala2005}) and the bound $\|P_{\II_r}EP_{\II_s}\|_2\leq \sqrt{|\II_r|\wedge |\II_s|}\|P_{\II_r}EP_{\II_s}\|_\infty$, we get that for all $t\geq 1$, \eqref{EqBoundCoeffBlock} is satisfied with probability at least $1-18\exp(-c_1t)$, provided that we choose
\[
(a_1,a_2,a_3)=(\epsilon \sqrt{t},\epsilon \sqrt{t},\epsilon \sqrt{t}),\quad (b_1,b_2,b_3)=(C_2\epsilon \sqrt{t},C_2\epsilon\sqrt{t} (k-1),C_2\epsilon\sqrt{t}(p-k)).
\]
By Theorem \ref{ThmSinThetaAbstr}, we get, for all $t\geq 1$, with probability at least $1-18\exp(-c_1t)$,
\[
\|\hat P_1-P_1\|_2^2\leq C_3\bigg(\frac{\epsilon^2tk}{(\lambda_1-\lambda_2)^2}+\frac{\epsilon^2 t(p-k)}{\lambda_1^2}+\frac{\epsilon^4t^2k^2}{(\lambda_1-\lambda_2)^4}+\frac{\epsilon^4t^2(p-k)^2}{\lambda_1^2(\lambda_1-\lambda_2)^2}\bigg),
\]
provided that
\begin{equation}\label{EqRademacher3}
\frac{\epsilon^2tk}{(\lambda_1-\lambda_2)^2}+\frac{\epsilon^2t(p-k)}{\lambda_1(\lambda_1-\lambda_2)}\leq c_2.
\end{equation}
Since always $\|\hat P_1-P_1\|_2^2\leq 2$, Condition \eqref{EqRademacher3} can be dropped and the above bound can be rearranged into the desired form \eqref{EqGOEPert}, by adjusting $C_3$.

\subsubsection*{Spiked covariance model}
Consider the empirical covariance operator from Section \ref{SecECO}. Let $\HH=\mathbb{R}^p$ and let $\Sigma=\mu_1P_{\II_1}+\mu_2P_{\II_2}+\mu_3P_{\II_3}$ with $\mu_1>\mu_2>\mu_3>0$ and $m_r=|\II_r|$, $r=1,2,3$. Assume that the $\eta_j$ are independent and sub-Gaussian, meaning that for some constant $C_\eta>0$, $\mathbb{E}^{1/q}|\eta_j|^q\leq C_\eta\sqrt{q}$ for all natural numbers $q\geq 1$ and all $j=1,\dots,p$. Then Theorem \ref{ThmSinThetaAbstr} yields that for all $t\geq 1$, with probability at least $1-9\exp(-(m_1\wedge m_2\wedge m_3)  t)$,
\begin{align}
&\|\hat P_{\II_1}-P_{\II_1}\|_2^2\leq Cm_1\bigg(\frac{\mu_1^2k}{(\mu_1-\mu_2)^2}\frac{t}{n}+\frac{\mu_1^2(p-k)}{(\mu_1-\mu_3)^2}\frac{t}{n}+\frac{\mu_1^4(p-k)^2}{(\mu_1-\mu_2)^2(\mu_1-\mu_3)^2}\frac{t^2}{n^2}\bigg).\label{EqSCM}
\end{align}
where $k=m_1+m_2$ is the dimension of the spiked part and $C>0$ is a constant depending only on $C_\eta$. In comparison, the Davis-Kahan $\sin\Theta$ theorem in \eqref{dk} with $\|E\|_2$ replaced by $\sqrt{|I|}\|E\|_\infty$ yields a bound of order $\mu_1^2m_1p/(n(\mu_1-\mu_2)^2)$.

Let us deduce \eqref{EqSCM} from Theorem \ref{ThmSinThetaAbstr}, by replacing Lemma \ref{prop:fuknagaev:proj} with the following concentration inequality for the operator norm of empirical covariance operators.
\begin{lemma} In the above setting, there is a constant $C_1>0$ depending only on $C_\eta$ such that for all $r,s=1,2,3$ and all $t\geq 1$ satisfying $t(m_r\vee m_s)/n\le 1$, we have
\[
\mathbb{P}\bigg(\|P_{\II_r}EP_{\II_s}\|_\infty>C_1\sqrt{\frac{\mu_r\mu_s(m_r\vee m_s)t}{n}}\bigg)\leq \exp(-(m_r\vee m_s) t).
\]
\end{lemma}
The case $r=s$ follows from \cite[Theorem 1]{KL14}, the non-diagonal case follows from a similar standard net argument as presented therein (see also \cite{V12,OROURKE201826}). Proceeding as in the proof of \eqref{EqGOEPert}, we get that for all $t\geq 1$, \eqref{EqBoundCoeffBlock} is satisfied with probability at least $1-9\exp(-(m_1\wedge m_2\wedge m_3)  t)$, provided that we choose
\begin{align*}
(a_1,a_2,a_3)&=(\mu_1\sqrt{t/n},\mu_2\sqrt{t/n}, \mu_3\sqrt{t/n}),\\
(b_1,b_2,b_3)&=(C_1\mu_1m_1\sqrt{t/n},C_1\mu_2m_2\sqrt{t/n} ),C_1\mu_3m_3\sqrt{t/n}).
\end{align*}
Applying Theorem \ref{ThmSinThetaAbstr}, \eqref{EqSCM} follows from a similar computation leading to \eqref{EqGOEPert}.

\section{Proof of main theorem}\label{SecProof}
\subsection{Separation of eigenvalues}
In this section, we show that under Condition \eqref{EqCCondBlock}, the empirical eigenvalues $(\hat\lambda_i)_{i\in \II}$ are well-separated from $(\lambda_j)_{j\notin \II}$.

\begin{lemma}\label{LemEVConcBlock}
Under the assumptions of Theorem \ref{ThmSinThetaAbstr}, we have
\[
|\hat\lambda_i-\lambda_j|\geq \frac{|\lambda_i-\lambda_j|}{2}
\]
for all $i\in\II$ and $j\notin\II$.
\end{lemma}

The proof is based on the following result, which is an intermediate step in the proof of Theorems 3.7 and 3.10 in \cite{RW17}. In fact, \eqref{EqEvRD}, for instance, follows from the min-max characterisation of eigenvalues in combination with \cite[Lemma 3.8]{RW17}.
\begin{proposition}\label{PropEvRC}
For all $i\geq 1$ and $y>0$, we have the implications
\begin{equation}\label{EqEvRD}
\Big\|\sum_{k\geq i}\sum_{l\geq i}\frac{1}{\lambda_i+y-\lambda_k}\frac{1}{\lambda_i+y-\lambda_l}P_kEP_l\Big\|_\infty^2\leq 1\Rightarrow
\hat \lambda_i-\lambda_i\leq y
\end{equation}
and
\begin{equation}\label{EqEvLDOp}
\Big\|\sum_{k\leq i}\sum_{l\leq i}\frac{1}{\sqrt{\lambda_k+y-\lambda_i}}\frac{1}{\sqrt{\lambda_l+y-\lambda_i}}P_kEP_l\Big\|_\infty^2\leq 1\Rightarrow
 \hat\lambda_i-\lambda_i\geq -y.
\end{equation}
\end{proposition}

\begin{proof}[Proof of Lemma \ref{LemEVConcBlock}]
It suffices to show that
\begin{equation}\label{EqLD}
\hat\lambda_i-\lambda_j\geq \frac{\lambda_i-\lambda_j}{2}
\end{equation}
for all $i\in\II$ and $j\notin\II$ such that $i< j$, and that
\begin{equation}\label{EqRD}
\lambda_j-\hat\lambda_i\geq \frac{\lambda_j-\lambda_i}{2}
\end{equation}
for all $i\in\II$ and $j\notin\II$ such that $j< i$. We only prove \eqref{EqLD}, since the proof of \eqref{EqRD} follows the same line of arguments. First, \eqref{EqLD} is equivalent to
\begin{equation}
\hat\lambda_i-\lambda_i\geq -\frac{\lambda_i-\lambda_j}{2}
\end{equation}
for all $i\in\II$ and $j\notin\II$ such that $i< j$. Thus, it suffices to show that the left-hand side in \eqref{EqEvLDOp} is satisfied with $y=(\lambda_i-\lambda_j)/2$. For $r\geq 1$, set
\[T_r=\sum_{k\in\II_r,k\leq i}\frac{1}{\sqrt{\lambda_k+y-\lambda_i}}P_k.\]
(Set $T_r=0$ if the summation is empty.) Using that the $T_r$ are self-adjoint and have orthogonal ranges, we have
\begin{align}
&\Big\|\sum_{k\leq i}\sum_{l\leq i}\frac{1}{\sqrt{\lambda_k+y-\lambda_i}}\frac{1}{\sqrt{\lambda_l+y-\lambda_i}}P_kEP_l\Big\|_\infty^2\label{EqOpNormOrth}\\
&=\Big\|\sum_{r\geq 1}\sum_{s\geq 1}T_rET_s\Big\|_\infty^2\leq \sum_{r\geq 1}\sum_{s\geq 1}\|T_rET_s\|_\infty^2.\nonumber
\end{align}
Using the identities $T_r=T_rP_{\II_r}=P_{\II_r}T_r$, the fact that the operator norm is sub-multiplicative, and \eqref{EqBoundCoeffBlock}, we have
\begin{equation*}
\|T_rET_s\|_\infty^2 \leq (a_rb_s+b_ra_s)\|T_r\|_\infty^2\|T_s\|_\infty^2
\end{equation*}
for all $r,s\geq 1$. Hence,
\begin{equation}\label{EqOPCondEv}
\sum_{r\geq 1}\sum_{s\geq 1}\|T_rET_s\|_\infty^2\leq2 \bigg( \sum_{r\geq 1}a_r\|T_r\|_\infty^2\bigg)\bigg( \sum_{s\geq 1}b_s\|T_s\|_\infty^2\bigg).
\end{equation}
Now, using that $\min_{k\in\II_r,k\leq i}(\lambda_k+y-\lambda_i)\geq \min_{k\in\II_r}|\lambda_k-\lambda_i|\geq g_r$ for $r> \nb$, and $\min_{k\in\II_r,k\leq i}(\lambda_k+y-\lambda_i)\geq \min_{k\in\II_r}|\lambda_k-\lambda_j|/2\geq g_r/2$ for $r\leq \nb$, we obtain that
\begin{align}\label{EqSGTPo1}
 \|T_r\|_\infty^2\leq 2/g_r
\end{align}
for all $r\geq 1$. Using \eqref{EqOpNormOrth}-\eqref{EqSGTPo1} in combination with \eqref{EqCCondBlock}, we conclude that
\begin{align*}
&\Big\|\sum_{k\leq i}\sum_{l\leq i}\frac{1}{\sqrt{\lambda_k+y-\lambda_i}}\frac{1}{\sqrt{\lambda_l+y-\lambda_i}}P_kEP_l\Big\|_\infty^2\leq 1/8\leq 1
\end{align*}
and the claim follows from \eqref{EqEvLDOp}.
\end{proof}

\subsection{Key contraction phenomenon}

In this section, we present our main lemma:
\begin{lemma}\label{KeyLemmaBlock}
Under the assumptions of Theorem \ref{ThmSinThetaAbstr}, the inequality
\begin{align*}
\bigg\|\sum_{i\in\II}\frac{1}{\hat\lambda_i-\lambda_{j}}P_{\II_r}E\hat P_i\bigg\|_2\leq \bigg(\frac{3}{2}\sqrt{b_r}+4\sqrt{a_r}\bigg(\sum_{s\geq 1}\frac{b_s}{g_s}\bigg)\bigg)\sqrt{\sum_{s\leq \nb}\frac{b_s}{\min_{i\in\II_s}(\lambda_{i}- \lambda_j)^2}}
\end{align*}
holds for all $r\geq 1$ and all $j\notin \II$.
\end{lemma}

\begin{remark}
If the partitions consist of singletons and $a_j=b_j=x\lambda_j$ (as in Theorem \ref{ThmSinTheta}), then Lemma \ref{KeyLemmaBlock} implies
\begin{align*}
\bigg\|\sum_{i\in\II}\frac{1}{\hat\lambda_i-\lambda_{j}}P_{k}E\hat P_i\bigg\|_2\leq 2x\sqrt{\sum_{i\in\II}\frac{\lambda_i\lambda_k}{(\lambda_{i}- \lambda_j)^2}}
\end{align*}
for all $k\ge 1$ and all $j\notin \II$. Ignoring the constant $2$, the right-hand side supplies an upper bound for the Hilbert-Schmidt norm of $\sum_{i\in\II}P_{k}E P_i/(\lambda_i-\lambda_{j})$ using \eqref{EqCoeffRelBound}.
\end{remark}

\begin{proof}
We recall some simple properties of the Hilbert-Schmidt norm which we will use in the proof without further comment. For Hilbert-Schmidt operators $A$ and $B$ on $\HH$, we have $\|AB\|_2\leq \|A\|_{\infty}\|B\|_2$. Moreover, for a Hilbert-Schmidt operator $A$ on $\HH$ and a bounded sequence of real numbers $(x_i)_{i\geq 1}$ we have $\|\sum_{i\geq 1}x_iP_iA\|_2^2=\sum_{i\geq 1}x_i^2\|P_iA\|_2^2$ and the same identity holds for $P_iA$ replaced by $AP_i$.

Let $r\geq 1$ be arbitrary. By the identity $I=P_{\II}+P_{\II^c}$ (see the convention in Section \ref{SecNotation}), and the triangular inequality, we have
\begin{align}
\bigg\|\sum_{i\in \II}\frac{1}{\hat\lambda_i-\lambda_{j}}P_{\II_r}E\hat P_i\bigg\|_2\label{EqTriangDec}
&\leq  \bigg\|\sum_{i\in \II}\frac{1}{\hat\lambda_i-\lambda_{j}}P_{\II_r}EP_{\II}\hat P_i\bigg\|_2\\
&+\bigg\|\sum_{i\in \II}\frac{1}{\hat\lambda_i-\lambda_{j}}P_{\II_r}EP_{\II^c}\hat P_i\bigg\|_2
\nonumber.
\end{align}
Note that all denominators are non-zero by Lemma~\ref{LemEVConcBlock}. We start with the first term on the right-hand side of \eqref{EqTriangDec}. By the identities
\begin{align}
(\hat \lambda_i-\lambda_{k})P_{k}\hat P_i&=P_{k}E \hat  P_i,\label{EqPertBB}\\
\frac{1}{\hat\lambda_i-\lambda_j}-\frac{1}{\lambda_{k}-\lambda_j}&=\frac{\lambda_{k}-\hat\lambda_i}{(\hat\lambda_i-\lambda_j)(\lambda_{k}-\lambda_j)}\nonumber,
\end{align}
where the latter holds provided that $i\in\II$ and $\lambda_{j}\neq \lambda_k$ (note that the denominators are non-zero by Lemma \ref{LemEVConcBlock}), we get
\begin{align*}
\sum_{i\in \II}\frac{1}{\hat\lambda_i-\lambda_j}P_{\II}\hat P_i-\sum_{k\in \II}\frac{1}{\lambda_{k}- \lambda_j}P_{k}\hat P_{\II}
&=  \sum_{k\in \II}\sum_{i\in \II}\bigg(\frac{1}{\hat\lambda_i-\lambda_j}-\frac{1}{\lambda_{k}- \lambda_j}\bigg)P_{k}\hat P_{i}\\
&=-\sum_{k\in \II}\sum_{i\in\II}\frac{1}{\lambda_{k}-\lambda_j}\frac{1}{\hat\lambda_i-\lambda_j}P_{k} E \hat P_{i}.
\end{align*}
Using this identity and the triangular inequality, we get
\begin{align}
\bigg\|\sum_{i\in \II}\frac{1}{\hat\lambda_i-\lambda_{j}}P_{\II_r}E P_{\II}\hat P_i\bigg\|_2\label{EqSwappTrick}
 &\leq \bigg\|\sum_{k\in \II}\frac{1}{\lambda_{k}- \lambda_{j}}P_{\II_r}EP_{k}\hat P_{\II}\bigg\|_2\\
 &+\bigg\|\sum_{k\in \II}\sum_{i\in\II}\frac{1}{\lambda_{k}-\lambda_{j}}\frac{1}{\hat\lambda_i-\lambda_{j}}P_{\II_r}EP_{k} E \hat P_{i}\bigg\|_2\nonumber.
\end{align}
The first term on the right-hand side of \eqref{EqSwappTrick} is bounded as follows:
\begin{align}
\bigg\|\sum_{k\in \II}\frac{1}{\lambda_{k}- \lambda_j}P_{\II_r}EP_{k}\hat P_{\II}\bigg\|_2
&\leq \bigg\|\sum_{k\in \II}\frac{1}{\lambda_{k}- \lambda_{j}}P_{\II_r}EP_{k}\bigg\|_2\label{EqKeyLemma1}.
\end{align}
Next, consider the second term on the right-hand side of \eqref{EqSwappTrick}. Using the triangular inequality, we have
\begin{align}
&\bigg\|\sum_{k\in \II}\sum_{i\in\II}\frac{1}{\lambda_{k}-\lambda_{j}}\frac{1}{\hat\lambda_i-\lambda_{j}}P_{\II_r}EP_{k} E \hat P_{i}\bigg\|_2\label{EqKeyLemma2}\\
&\leq\sum_{s\leq \nb}\bigg\|\sum_{k\in \II_s}\sum_{i\in\II}\frac{1}{\lambda_{k}-\lambda_{j}}\frac{1}{\hat\lambda_i-\lambda_{j}}P_{\II_r}EP_{k} E \hat P_{i}\bigg\|_2\nonumber\\
&=\sum_{s\leq \nb}\sqrt{\sum_{i\in\II} \frac{1}{(\hat\lambda_i-\lambda_{j})^2}\bigg\|\sum_{k\in \II_s}\frac{1}{\lambda_{k}-\lambda_{j}}P_{\II_r}EP_{k} E \hat P_{i}\bigg\|_2^2}.\nonumber
\end{align}
Now, for each $s\leq \nb$ and $i\in \II$,
\begin{align*}
&\sum_{k\in \II_s}\frac{1}{\lambda_{k}-\lambda_{j}}P_{\II_r}EP_{k} E \hat P_{i}=P_{\II_r}EP_{\II_s}\bigg(\sum_{k\in \II_s}\frac{1}{\lambda_{k}-\lambda_{j}}P_k\bigg)P_{\II_s}E\hat P_i
\end{align*}
and by \eqref{EqBoundCoeffBlock} and the definition of $g_s$, this implies that
\begin{align*}
\bigg\|\sum_{k\in \II_s}\frac{1}{\lambda_{k}-\lambda_{j}}P_{\II_r}EP_{k} E \hat P_{i}\bigg\|_2
&\leq \bigg\|P_{\II_r}EP_{\II_s}\bigg\|_\infty\bigg\|\sum_{k\in \II_s}\frac{1}{\lambda_{k}-\lambda_{j}}P_k\bigg\|_\infty\bigg\|P_{\II_s}E\hat P_i\bigg\|_2\\
&\leq \bigg(\frac{\sqrt{a_rb_s}}{g_s}+\frac{\sqrt{b_ra_s}}{g_s}\bigg)\bigg\|P_{\II_s}E\hat P_i\bigg\|_2.
\end{align*}
Inserting this inequality into \eqref{EqKeyLemma2}, we get
\begin{align}
&\bigg\|\sum_{k\in \II}\sum_{i\in\II}\frac{1}{\lambda_{k}-\lambda_{j}}\frac{1}{\hat\lambda_i-\lambda_{j}}P_{\II_r}EP_{k} E \hat P_{i}\bigg\|_2\label{EqKeyLemma3}\\
&\leq\sum_{s\leq \nb}\bigg(\frac{\sqrt{a_rb_s}}{g_s}+\frac{\sqrt{b_ra_s}}{g_s}\bigg)\bigg\|\sum_{i\in\II}\frac{1}{\hat\lambda_i-\lambda_{j}}P_{\II_s}E\hat P_i\bigg\|_2.\nonumber
\end{align}
For the second term on the right-hand side of \eqref{EqTriangDec} we proceed similarly. By \eqref{EqPertBB} and the triangular inequality, we have
\begin{align}
\bigg\|\sum_{i\in \II}\frac{1}{\hat\lambda_i-\lambda_{j}}P_{\II_r}E P_{\II^c}\hat P_i\bigg\|_2
&= \bigg\|\sum_{k\notin\II}\sum_{i\in \II}\frac{1}{\hat\lambda_{i}-\lambda_{k}}\frac{1}{\hat\lambda_i-\lambda_{j}}P_{\II_r}E P_{k}E\hat P_i\bigg\|_2\nonumber\\
&\leq\sum_{s>\nb}\sqrt{\sum_{i\in \II}\frac{1}{(\hat\lambda_i-\lambda_{j})^2}\bigg\|\sum_{k\in\II_s}\frac{1}{\hat\lambda_{i}-\lambda_{k}}P_{\II_r}E P_{k}E\hat P_i\bigg\|_2^2}.\label{EqKeyLemma4}
\end{align}
Now, for $s>\nb$ and $i\in\II$, we have
\begin{align*}
\sum_{k\in\II_s}\frac{1}{\hat\lambda_{i}-\lambda_{k}}P_{\II_r}E P_{k}E\hat P_i=P_{\II_r}EP_{\II_s}\bigg(\sum_{k\in\II_s}\frac{1}{\hat\lambda_{i}-\lambda_{k}}P_k\bigg)P_{\II_s}E\hat P_i
\end{align*}
and by \eqref{EqBoundCoeffBlock} and Lemma \ref{LemEVConcBlock}, this implies that
\begin{align*}
\bigg\|\sum_{k\in\II_s}\frac{1}{\hat\lambda_{i}-\lambda_{k}}P_{\II_r}E P_{k}E\hat P_i\bigg\|_2\leq2\bigg(\frac{\sqrt{a_rb_s}}{g_s}+\frac{\sqrt{b_ra_s}}{g_s}\bigg)\bigg\|P_{\II_s}E\hat P_i\bigg\|_2.
\end{align*}
Inserting this into \eqref{EqKeyLemma4}, we conclude that
\begin{align}
&\bigg\|\sum_{i\in \II}\frac{1}{\hat\lambda_i-\lambda_{j}}P_{\II_r}E P_{\II^c}\hat P_i\bigg\|_2\label{EqKeyLemma5}\\
&\leq 2\sum_{s>\nb}\bigg(\frac{\sqrt{a_rb_s}}{g_s}+\frac{\sqrt{b_ra_s}}{g_s}\bigg)\bigg\|\sum_{i\in\II}\frac{1}{\hat\lambda_i-\lambda_{j}}P_{\II_s}E\hat P_i\bigg\|_2\nonumber.
\end{align}
Collecting \eqref{EqTriangDec}, \eqref{EqSwappTrick}-\eqref{EqKeyLemma5}, we conclude that
\begin{align}
&\bigg\|\sum_{i\in \II}\frac{1}{\hat\lambda_i-\lambda_j}P_{\II_r}E\hat P_i\bigg\|_2
\leq \bigg\|\sum_{i\in \II}\frac{1}{\lambda_{i}- \lambda_j}P_{\II_r}EP_{i}\bigg\|_2\label{EqKeyLemma6}\\
&+2\sum_{s\geq 1}\bigg(\frac{\sqrt{a_rb_s}}{g_s}+\frac{\sqrt{b_ra_s}}{g_s}\bigg)\bigg\|\sum_{i\in\II}\frac{1}{\hat\lambda_i-\lambda_{j}}P_{\II_s}E\hat P_i\bigg\|_2\nonumber\quad \forall r\geq 1.
\end{align}
It remains to solve this recursive inequality. First, using \eqref{EqBoundCoeffBlock}, we have
\begin{align}
&\bigg\|\sum_{i\in \II}\frac{1}{\lambda_{i}- \lambda_j}P_{\II_r}EP_{i}\bigg\|_2 =\sqrt{\sum_{i\in \II}\frac{\|P_{\II_r}EP_{i}\|_2^2}{(\lambda_{i}- \lambda_j)^2}}\label{EqKeyLemma8}\\
&\leq \sqrt{\sum_{s\leq \nb}\frac{\|P_{\II_r}EP_{\II_s}\|_2^2}{\min_{i\in\II_s}(\lambda_{i}- \lambda_j)^2}}
\leq \sqrt{\sum_{s\leq \nb}\frac{b_rb_s}{\min_{i\in\II_s}(\lambda_{i}- \lambda_j)^2}}=:\sqrt{b_r}B\nonumber
\end{align}
for all $r\geq 1$. If we set
\begin{equation*}
A_r=\bigg\|\sum_{i\in\II}\frac{1}{\hat\lambda_i-\lambda_{j}}P_{\II_r}E\hat P_i\bigg\|_2\qquad\forall r\geq 1,
\end{equation*}
then \eqref{EqKeyLemma6} implies that
\begin{equation}\label{EqRecIneq1}
A_r\leq \sqrt{b_r}B+2\sqrt{a_r}\bigg(\sum_{s\geq 1}\frac{\sqrt{b_s}}{g_s}A_s\bigg)+2\sqrt{b_r}\bigg(\sum_{s\geq 1}\frac{\sqrt{a_s}}{g_s}A_s\bigg)\quad\forall r\geq 1.
\end{equation}
Multiplying both sides with $\sqrt{b_r}/g_r$ and summing over $r\geq 1$, we have
\begin{align*}
&\sum_{r\geq 1}\frac{\sqrt{b_r}}{g_r}A_r\\
&\leq \bigg(\sum_{r\geq 1}\frac{b_r}{g_r}\bigg)B+2\bigg( \sum_{r\geq 1}\frac{\sqrt{a_rb_r}}{g_r}\bigg)\bigg(\sum_{s\geq 1}\frac{\sqrt{b_s}}{g_s}A_s\bigg)+2\bigg( \sum_{r\geq 1}\frac{b_r}{g_r}\bigg)\bigg(\sum_{s\geq 1}\frac{\sqrt{a_s}}{g_s}A_s\bigg).
\end{align*}
By \eqref{EqCCondBlock} and the Cauchy Schwarz inequality, this implies
\begin{align*}
\sum_{r\geq 1}\frac{\sqrt{b_r}}{g_r}A_r\leq \frac{4}{3}\bigg(\sum_{r\geq 1}\frac{b_r}{g_r}\bigg)B+\frac{8}{3}\bigg( \sum_{r\geq 1}\frac{b_r}{g_r}\bigg)\bigg(\sum_{s\geq 1}\frac{\sqrt{a_s}}{g_s}A_s\bigg).
\end{align*}
Inserting this inequality into \eqref{EqRecIneq1}, we get
\begin{align}
A_r&\leq \sqrt{b_r}B+\frac{8}{3}\sqrt{a_r}\bigg(\sum_{s\geq 1}\frac{b_s}{g_s}\bigg)B\label{EqRecIneq3}\\
&+\frac{16}{3}\sqrt{a_r}\bigg( \sum_{s\geq 1}\frac{b_s}{g_s}\bigg)\bigg(\sum_{s\geq 1}\frac{\sqrt{a_s}}{g_s}A_s\bigg)+2\sqrt{b_r}\bigg(\sum_{s\geq 1}\frac{\sqrt{a_s}}{g_s}A_s\bigg)\quad\forall r\geq 1.\nonumber
\end{align}
Now, multiplying both sides with $\sqrt{a_r}/g_r$ and summing over $r\geq 1$, we have
\begin{align*}
&\sum_{r\geq 1}\frac{\sqrt{a_r}}{g_r}A_r \leq \bigg(\sum_{r\geq 1}\frac{\sqrt{a_rb_r}}{g_r}\bigg)B+\frac{8}{3}\bigg(\sum_{r\geq 1}\frac{a_r}{g_r}\bigg)\bigg(\sum_{s\geq 1}\frac{b_s}{g_s}\bigg)B\\
&+\frac{16}{3}\bigg(\sum_{r\geq 1}\frac{a_r}{g_r}\bigg)\bigg( \sum_{s\geq 1}\frac{b_s}{g_s}\bigg)\bigg(\sum_{s\geq 1}\frac{\sqrt{a_s}}{g_s}A_s\bigg)+2\bigg( \sum_{r\geq 1}\frac{\sqrt{a_rb_r}}{g_r}\bigg)\bigg(\sum_{s\geq 1}\frac{\sqrt{a_s}}{g_s}A_s\bigg).
\end{align*}
By \eqref{EqCCondBlock} and the Cauchy Schwarz inequality, this implies
\[
\sum_{r\geq 1}\frac{\sqrt{a_r}}{g_r}A_r\leq \frac{B}{8}+\frac{B}{24}+\frac{1}{12}\bigg(\sum_{s\geq 1}\frac{\sqrt{a_s}}{g_s}A_s\bigg)+\frac{1}{4}\bigg(\sum_{s\geq 1}\frac{\sqrt{a_s}}{g_s}A_s\bigg)
\]
and thus
\[
\sum_{r\geq 1}\frac{\sqrt{a_r}}{g_r}A_r\leq \frac{B}{4}.
\]
Inserting this into \eqref{EqRecIneq3}, we conclude that
\begin{align*}
A_r\leq \frac{3}{2}\sqrt{b_r}B+4\bigg(\sqrt{a_r}\sum_{s\geq 1}\frac{b_s}{g_s}\bigg)B\qquad\forall r\geq 1,
\end{align*}
and the claim follows from inserting the definitions of $A_r$ and $B$.
\end{proof}

\subsection{End of proof of Theorem \ref{ThmSinThetaAbstr}}
We have
\begin{equation}\label{EqHSAdd}
\|\hat{P}_{\mathcal{I}}-P_{\mathcal{I}}\|_2^2=2\langle I-P_{\mathcal{I}},\hat{P}_{\mathcal{I}}\rangle=2\sum_{r>\nb}\langle P_{\II_r},\hat P_{\II}\rangle=2\sum_{r>\nb}\|P_{\II_r}\hat P_{\II}\|_2^2.
\end{equation}
By \eqref{EqPertBB}, we have
\begin{equation}\label{EqExpBoundMin}
\|P_{\II_r}\hat P_{\II}\|_2^2=\sum_{j\in\II_r}\sum_{i\in\II}\frac{\|P_jE\hat P_i\|_2^2}{(\hat\lambda_i-\lambda_j)^2}\leq\sum_{i\in\II}\frac{\|P_{\II_r}E\hat P_i\|_2^2}{\min_{j\in\II_r}(\hat\lambda_i-\lambda_j)^2}.
\end{equation}
Note that all denominators are non-zero by Lemma \ref{LemEVConcBlock}. Now, using \eqref{EqLD}, \eqref{EqRD}, and the fact that $\II_r$ is an interval, we get that $\min_{j\in\II_r}(\hat\lambda_i-\lambda_j)^2$ is attained at at most two points, namely at the endpoints of $\II_r$. Hence, there are $j_0,j_1\in \II_r$ such that
\begin{align*}
\|P_{\II_r}\hat P_{\II}\|_2^2&\leq \bigg\|\sum_{i\in\II}\frac{1}{\hat\lambda_i-\lambda_{j_0}}P_{\II_r}E\hat P_i\bigg\|_2^2+\bigg\|\sum_{i\in\II}\frac{1}{\hat\lambda_i-\lambda_{j_1}}P_{\II_r}E\hat P_i\bigg\|_2^2.
\end{align*}
Inserting this into \eqref{EqHSAdd} and applying Lemma \ref{KeyLemmaBlock}, the claim follows from a simple computation, using the inequality $(y+z)^2\leq 4y^2/3+4z^2$.
\qed

\bibliographystyle{plain}
\bibliography{lit}

\begin{thebibliography}{10}

\bibitem{adamczak2015}
R.~Adamczak.
\newblock A note on the {H}anson-{W}right inequality for random vectors with
  dependencies.
\newblock {\em Electron. Commun. Probab.}, 20:no. 72, 13 pp, 2015.

\bibitem{anderson1963}
T.W. Anderson.
\newblock Asymptotic theory for principal component analysis.
\newblock {\em The Annals of Mathematical Statistics}, 34:122--148, 1963.

\bibitem{MR1477662}
R.~Bhatia.
\newblock {\em Matrix analysis}.
\newblock Springer-Verlag, New York, 1997.

\bibitem{MR3185193}
S.~Boucheron, G.~Lugosi, and P.~Massart.
\newblock {\em Concentration inequalities}.
\newblock Oxford University Press, Oxford, 2013.

\bibitem{C83}
F.~Chatelin.
\newblock {\em Spectral approximation of linear operators}.
\newblock Academic Press, New York, 1983.

\bibitem{dauxois_1982}
J.~Dauxois, A.~Pousse, and Y.~Romain.
\newblock Asymptotic theory for the principal component analysis of a vector
  random function: some applications to statistical inference.
\newblock {\em J. Multivariate Anal.}, 12:136--154, 1982.

\bibitem{MR0264450}
C.~Davis and W.~M. Kahan.
\newblock The rotation of eigenvectors by a perturbation. {III}.
\newblock {\em SIAM J. Numer. Anal.}, 7:1--46, 1970.

\bibitem{MR2434306}
U.~Einmahl and D.~Li.
\newblock Characterization of {LIL} behavior in {B}anach space.
\newblock {\em Trans. Amer. Math. Soc.}, 360:6677--6693, 2008.

\bibitem{gobet2004}
E.~Gobet, M.~Hoffmann, and M.~Reiß.
\newblock Nonparametric estimation of scalar diffusions based on low frequency
  data.
\newblock {\em Ann. Statist.}, 32:2223--2253, 2004.

\bibitem{HH09}
P.~Hall and M.~Hosseini-Nasab.
\newblock Theory for high-order bounds in functional principal components
  analysis.
\newblock {\em Math. Proc. Cambridge Philos. Soc.}, 146:225--256, 2009.

\bibitem{MR2920735}
L.~Horv\'ath and P.~Kokoszka.
\newblock {\em Inference for functional data with applications}.
\newblock Springer, New York, 2012.

\bibitem{HE15}
T.~Hsing and R.~Eubank.
\newblock {\em Theoretical foundations of functional data analysis, with an
  introduction to linear operators}.
\newblock John Wiley \& Sons, Ltd., Chichester, 2015.

\bibitem{I98}
I.~C.~F. Ipsen.
\newblock Relative perturbation results for matrix eigenvalues and singular
  values.
\newblock {\em Acta numerica}, 7:151--201, 1998.

\bibitem{I00}
I.~C.~F. Ipsen.
\newblock An overview of relative $\sin\theta$ theorems for invariant subspaces
  of complex matrices.
\newblock {\em J. Comput. Appl. Math.}, 123:131--153, 2000.

\bibitem{M16}
M.~Jirak.
\newblock Optimal eigen expansions and uniform bounds.
\newblock {\em Probab. Theory Related Fields}, 166:753--799, 2016.

\bibitem{JW18}
M.~Jirak and M.~Wahl.
\newblock Relative perturbation bounds with applications to empirical
  covariance operators.
\newblock Available at https://arxiv.org/pdf/1802.02869, 2018.

\bibitem{K95}
T.~Kato.
\newblock {\em Perturbation theory for linear operators}.
\newblock Springer-Verlag, Berlin, reprint of the 1980 edition, 1995.

\bibitem{KL16b}
V.~Koltchinskii and K.~Lounici.
\newblock Asymptotics and concentration bounds for bilinear forms of spectral
  projectors of sample covariance.
\newblock {\em Ann. Inst. Henri Poincar\'{e}}, 52:1976--2013, 2016.

\bibitem{KL14}
V.~Koltchinskii and K.~Lounici.
\newblock Concentration inequalities and moment bounds for sample covariance
  operators.
\newblock {\em Bernoulli}, 23:110--133, 2017.

\bibitem{latala2005}
R.~{Lata{\l}a}.
\newblock Some estimates of norms of random matrices.
\newblock {\em Proc. Amer. Math. Soc.}, 133(5):1273--1282, 2005.

\bibitem{mas_complex_2014}
A.~Mas and F.~Ruymgaart.
\newblock High-dimensional principal projections.
\newblock {\em Complex Anal. Oper. Theory}, 9:35--63, 2015.

\bibitem{OROURKE201826}
S.~O'Rourke, V.~Vu, and K.~Wang.
\newblock Random perturbation of low rank matrices: improving classical bounds.
\newblock {\em Linear Algebra Appl.}, 540:26--59, 2018.

\bibitem{RW17}
M.~Rei{\ss} and M.~Wahl.
\newblock Non-asymptotic upper bounds for the reconstruction error of {PCA}.
\newblock Available at https://arxiv.org/pdf/1609.03779v2.pdf, 2016.

\bibitem{SS2001}
B.~Scholkopf and A.~J. Smola.
\newblock {\em Learning with Kernels: Support Vector Machines, Regularization,
  Optimization, and Beyond}.
\newblock MIT Press, Cambridge, MA, USA, 2001.

\bibitem{vH15}
R.~{Van Handel}.
\newblock Structured random matrices.
\newblock In {\em Convexity and Concentration}, The IMA Volumes in Mathematics
  and its Applications, 161, pages 107--165. Springer, New York, 2017.

\bibitem{V12}
R.~Vershynin.
\newblock Introduction to the non-asymptotic analysis of random matrices.
\newblock In {\em Compressed sensing}, pages 210--268. Cambridge Univ. Press,
  Cambridge, 2012.

\bibitem{VVU2011}
V.~Vu.
\newblock Singular vectors under random perturbation.
\newblock {\em Random Structures Algorithms}, 39:526--538, 2011.

\bibitem{wang_samworth_biometrika_2015}
Y.~Yu, T.~Wang, and R.~J. Samworth.
\newblock A useful variant of the {D}avis–{K}ahan theorem for statisticians.
\newblock {\em Biometrika}, 102:315--323, 2015.

\end{thebibliography}

\end{document}